\documentclass{amsart}
\usepackage[a4paper]{geometry}

\usepackage{amsfonts,amsfonts,latexsym}
\usepackage{mathtools}

\usepackage{subcaption}
\usepackage[final]{listings}
\usepackage{csquotes}
\usepackage[cal=dutchcal]{mathalfa}

\usepackage{tikz}
\usepackage{tkz-graph}
\usetikzlibrary{arrows,shapes,positioning,decorations.markings,fit}

\usepackage{booktabs}

\usepackage{doi}
\usepackage{hyperref}

\usepackage[capitalize]{cleveref}

\crefname{lstlisting}{listing}{listings}
\Crefname{lstlisting}{Listing}{Listings}
\crefname{equation}{equation}{equations}

\newtheorem{theorem}{Theorem}[section]
\newtheorem{lemma}[theorem]{Lemma}
\newtheorem{definition}[theorem]{Definition}
\newtheorem{observation}{Observation}
\newtheorem{example}[theorem]{Example}

\usepackage{xfrac}

\newcommand*{\specname}[1]{\ensuremath{\mathbf{#1}}}
\newcommand*{\symgp}[1]{\ensuremath{\mathfrak{S}_{#1}}}
\newcommand*{\catname}[2][]{\ensuremath{\boldsymbol{\mathsf{#2}}}_{#1}}
\newcommand*{\fieldname}[1]{\ensuremath{\mathbf{#1}}}

\newcommand*{\specset}[1][]{\ensuremath{\specname{E}}_{#1}}
\newcommand*{\specsing}[1][X]{\ensuremath{\specname{#1}}}
\newcommand*{\speclin}[1][]{\ensuremath{\specname{Lin}_{#1}}}
\newcommand*{\speccyc}[1][]{\ensuremath{\specname{Cyc}_{#1}}}

\newcommand*{\spectree}[1][]{\ensuremath{\specname{Tree}_{#1}}}
\newcommand*{\specrtree}[1][]{\ensuremath{\specname{RTree}_{#1}}}
\newcommand*{\specrblt}[1][]{\ensuremath{\specname{RBLT}_{#1}}}
\newcommand*{\specbintree}[1][]{\ensuremath{\specname{BT}}_{#1}}
\newcommand*{\specbintreer}[1][]{\ensuremath{\specname{BTR}}_{#1}}

\newcommand*{\specgraph}[1][]{\ensuremath{\specname{Graph}_{#1}}}
\newcommand*{\specgraphc}[1][]{\ensuremath{\specname{GraphC}_{#1}}}
\newcommand*{\specgraphbc}[1][]{\ensuremath{\specname{GraphBC}_{#1}}}
\newcommand*{\specdigraph}[1][]{\ensuremath{\specname{Digraph}_{#1}}}
\newcommand*{\specdigraphc}[1][]{\ensuremath{\specname{DigraphC}_{#1}}}

\newcommand*{\specpoly}[1][]{\ensuremath{\specname{Poly}_{#1}}}
\newcommand*{\specpath}[1][]{\ensuremath{\specname{Path}_{#1}}}
\newcommand*{\specsubs}[1][]{\ensuremath{\specname{Subs}}_{#1}}

\newcommand*{\pointed}[2][]{\ensuremath{{#2}^{\bullet {#1}}}}

\newcommand*{\ci}[2][]{ \ensuremath{ Z_{#2}^{#1} } }
\newcommand*{\civars}[3][]{ \ensuremath{ Z_{#2}^{#1} \pbrac{#3} } }
\newcommand*{\gci}[3][]{ \ci[{#2} {#1}]{#3} }

\newcommand*{\gcielt}[4][]{ \ci[{#2} {#1}]{#3} \pbrac{#4} }
\newcommand*{\gcieltvars}[5][]{ \civars[{#2} {#1}]{#3}{#4} \pbrac{#5} }

\newcommand*{\orbits}[1]{\ensuremath{\Omega \pbrac*{#1}}}
\newcommand*{\posset}{\ensuremath{\fieldname{P}}}
\newcommand*{\poscols}[1]{\ensuremath{\posset^{#1}}}
\newcommand*{\intset}[1]{\ensuremath{\sbrac*{#1}}}
\newcommand*{\transport}[1]{\ensuremath{\sbrac*{#1}}}

\DeclarePairedDelimiter{\pbrac}{(}{)}
\DeclarePairedDelimiter{\sbrac}{[}{]}
\DeclarePairedDelimiter{\cbrac}{\{}{\}}
\DeclarePairedDelimiter{\abrac}{\langle}{\rangle}
\DeclarePairedDelimiter{\abs}{\lvert}{\rvert}
\DeclarePairedDelimiter{\floor}{\lfloor}{\rfloor}
\DeclarePairedDelimiter{\trans}{[}{]}

\DeclareMathOperator{\Aut}{Aut}
\DeclareMathOperator{\fix}{fix}
\DeclareMathOperator{\Fix}{Fix}
\DeclareMathOperator{\stab}{stab}

\newcommand*{\funccomp}{\mathbin{\square}}

\newcommand*{\quot}[2]{\ensuremath{\sfrac{#1}{#2}}}

\newcommand*{\subgp}{\ensuremath{\subseteq}}

\newcommand*{\code}[1]{\texttt{#1}}

\tikzstyle{every picture}+=[
  invisinode/.style={shape=circle,fill=white},
  dashbox/.style={dotted,shape=rectangle,draw},
  cnode/.style={shape=circle,fill=white,draw},
  cfnode/.style={shape=circle,fill=black,draw},
  snode/.style={shape=rectangle,fill=white,draw},
  dnode/.style={shape=diamond,fill=white,draw},
  edge/.style={draw,thick},
  diredge/.style={edge,->},
  rootedge/.style={edge,very thick,double},
  dirrootedge/.style={diredge,rootedge},
]

\begin{document}
\title{Species with an equivariant group action}
\author{Andrew Gainer-Dewar}
\address{
  Department of Mathematics and Computer Science,
  Hobart and William Smith Colleges,
  300 Pulteney Street,
  Geneva, New York, USA 06051
}
\curraddr{UConn Health Center for Quantitative Medicine, 195 Farmington Avenue, Farmington, CT, USA}
\email{andrew.gainer.dewar@gmail.com}

\maketitle

\begin{abstract}Joyal's theory of combinatorial species provides a rich and elegant framework for enumerating combinatorial structures by translating structural information into algebraic functional equations.
  We also extend the theory to incorporate information about ``structural'' group actions (i.e.~those which commute with the label permutation action) on combinatorial species, using the $\Gamma$-species of Henderson, and present P\'{o}lya-theoretic interpretations of the associated formal power series for both ordinary and $\Gamma$-species.
  We define the appropriate operations $+$, $\cdot$, $\circ$, and $\square$ on $\Gamma$-species, give formulas for the associated operations on $\Gamma$-cycle indices, and illustrate the use of this theory to study several important examples of combinatorial structures.
  Finally, we demonstrate the use of the Sage computer algebra system to enumerate $\Gamma$-species and their quotients.
\end{abstract}

\section{Preliminaries}
\label{sec:prelim}

\subsection{Classical P\'{o}lya theory}
\label{sec:polya}
We recall here some classical results of P\'{o}lya theory for convenience.

Let $\Lambda$ denote the ring of abstract symmetric functions and $p_{i}$ the elements of the power-sum basis of $\Lambda$.
Further, let $P$ denote the ring of formal power series in the family of indeterminates $x_{1}, x_{2}, \dots$, and let $\eta: \Lambda \to P$ denote the map which expands each symmetric function in the underlying $x$-variables.

Let $G$ be a finite group which acts on a finite set $S$ of cardinality $n$.
The \emph{classical cycle index polynomial} of the action of $G$ on $S$ is the power series
\begin{equation}
  \label{eq:pcisdef}
  \ci{G} \pbrac{p_{1}, p_{2}, \dots, p_{n}} = \frac{1}{\abs{G}} \sum_{\sigma \in G} p_{\sigma}
\end{equation}
where $p_{\sigma} = p_{1}^{\sigma_{1}} p_{2}^{\sigma_{2}} \dots$ for $\sigma_{i}$ the number of $i$-cycles of the action of $\sigma$ on $S$.
(In particular, if $G \subgp \symgp{n}$, we frequently consider the action of $G$ on $\intset{n}$ as permutations; then $p_{\sigma}$ simply counts the $i$-cycles in $\sigma$ as a permutation.)

In this language, the celebrated P\'{o}lya enumeration theorem then has a simple form:
\begin{theorem}[P\'{o}lya enumeration theorem]
  \label{thm:polya}
  Let $\ci{G}$ be the classical cycle index polynomial of a fixed action of the finite group $G$ on the finite set $S$ and let $\pi = \abrac{\pi_{1}, \dots, \pi_{k}}$ be a vector of positive integers summing to $n$.
  Then the number of $G$-orbits of colorings of $S$ having $\pi_{i}$ instances of color $i$ is equal to the coefficient of $x_{\pi} = x_{1}^{\pi_{1}} x_{2}^{\pi_{2}} \dots x_{k}^{\pi_{k}}$ in $\eta \pbrac{\ci{G}}$.
\end{theorem}

\subsection{Combinatorial species}
\label{sec:specprelim}
The theory of combinatorial species, introduced by Andr\'{e} Joyal in \cite{joy:species}, provides an elegant framework for understanding the connection between classes of combinatorial structures and their associated counting series.
We adopt the categorical perspective on species; the reader unfamiliar with these constructs should first consult the \enquote{species book} \cite{bll} for a primer on the associated combinatorics.

Let $\catname{FinSet}$ denote the category of finite sets with set maps and $\catname{FinBij}$ denote its \enquote{core}, the groupoid\footnote{Recall that a \emph{groupoid} is a category whose morphisms are all isomorphisms.} of finite sets with bijections.
A \emph{combinatorial species} $F$ is then a functor $F: \catname{FinSet} \to \catname{FinBij}$.
Specifically, $F$ carries each set $A$ of \enquote{labels} to the set $F \sbrac{A}$ of \enquote{$F$-structures labeled by $A$}, and each permutation $\sigma: A \to A$ to a permutation $F \sbrac{\sigma}: F \sbrac{A} \to F \sbrac{A}$.
(Thus, for example, for the species $\specgraph$ of graphs labeled at vertices, a permutation $\sigma \in \symgp{4}$ is transported to a permutation $\specgraph \sbrac{\sigma}$ on the class of labeled graphs with four vertices.)
The crucial combinatorial insight of species theory is that, for enumerative purposes, it is the algebraic structure of the group $F \sbrac{\symgp{A}}$ of \enquote{relabelings of $F$-structures over $A$} which is important, and \emph{not} the combinatorial details of the $F$-structures themselves.

Associated to each species $F$ are several formal power series which enumerate various sorts of $F$-structures.
Classically, the generating functions for labeled and unlabeled $F$-structures have received the most attention; species-theoretic analysis instead uses the \emph{cycle index series}, given by
\begin{equation}
  \label{eq:cidef}
  \ci{F} = \sum_{n \geq 0} \frac{1}{n!} \sum_{\sigma \in \symgp{n}} \fix \pbrac*{F \sbrac{\sigma}} p_{1}^{\sigma_{1}} p_{2}^{\sigma_{2}} \dots
\end{equation}
where $\sigma_{i}$ is the number of $i$-cycles in $\sigma$ and $p_{i}$ is a formal indeterminate.
It is easily shown (cf.~\cite{bll}) that we can recover the generating functions for labeled and unlabeled $F$-structures from $\ci{F}$, essentially by applying Burnside's lemma to the actions of $\symgp{n}$; however, the algebra of cycle index series captures the calculus of combinatorial structures more fully than that of generating functions.
Thus, we generally work at the cycle-index level until we have characterized a species of interest, then extract the desired enumerations.

It is often meaningful to speak of maps between combinatorial classes; for example, there is a natural map from the class $\spectree$ of trees to the class $\specgraph$ of simple graphs which simply interprets each tree as a graph.
Indeed, this map is \enquote{natural} in the sense that it respects the structure of the trees and is not dependent on labelings; this can be captured either by saying that it acts on \enquote{unlabeled trees and graphs} or by noting that it commutes with the actions of $\symgp{n}$ on labels.
Since $\spectree$ and $\specgraph$ are each functors, it turns out that this \enquote{naturality} condition is equivalent to the category-theoretic notion of a \emph{natural transformation}.
We can then define the category $\catname{Spec}$ of species as simply the functor category $\catname{FinBij}^{\catname{FinSet}}$.
As noted in \cite[\S 1.1]{argthesis}, the epi-, mono-, and isomorphisms of this category have natural combinatorial interpretations as \enquote{species embeddings}, \enquote{species coverings}, and \enquote{species isomorphisms}.
(Of course, this sort of category-theoretic approach obscures the combinatorial applications of the theory, but the compactness of the representation is attractive, and it suggests that this is a \enquote{natural} structure.)

\subsection{P\'{o}lya theory for species}
\label{sec:polyaspec}
Once again, let $F$ be a combinatorial species and let $\ci{F}$ be its cycle index series.
The formal indeterminates $p_{i}$ in \cref{eq:cidef} may be interpreted as the elements of the power-sum basis of the ring $\Lambda$ of abstract symmetric functions introduced in \cref{sec:polya}.
To demonstrate the usefulness of this interpretation, we note that the P\'{o}lya-theoretic cycle index polynomial of \cref{eq:pcisdef} and the species-theoretic cycle index series of \cref{eq:cidef} are intimately related.

\begin{lemma}[\protect{\cite[\S~3.2.1, Prop.~13, eq.~3]{joy:species}}]
  \label{thm:cisrel}
  Let $F$ be a combinatorial species.
  Denote by $\orbits{F}$ the collection of orbits\footnote{Note that $\orbits{F}$ corresponds to the molecular decomposition of $F$.} of $F$-structures under the actions of the symmetric groups.
  For each such orbit $\omega \in \orbits{F}$, let $\stab \omega$ be the subgroup of $\symgp{n}$ which fixes\footnote{For a given $\omega$, all available choices of subgroup will be conjugate, and so the formula will not be affected by this choice.} some element of $\omega$.
  Then
  \begin{equation}
    \label{eq:cisrel}
    \ci{F} = \sum_{\omega \in \orbits{F}} \ci{\stab \omega}.
  \end{equation}
\end{lemma}

Thus, we may reasonably hope to extend the classical P\'{o}lya theory which results from \cref{thm:polya} to the species-theoretic context.
Typical species-theoretic analysis requires treating all structures as labeled and considering the orbits of structures under the actions of symmetric groups on those labels.
To connect this idea with the P\'{o}lya-theoretic idea of colors, we introduce an intermediate notion.

\begin{definition}
  \label{def:plstruct}
  Let $F$ be a combinatorial species and fix a positive integer $n$.
  Let $\posset$ denote the set of positive integers and let $\poscols{n}$ denote the set of \emph{colorings} $c: \posset \to \intset{n}$.
  Let $\symgp{n}$ act on $\poscols{n}$ by $\pbrac*{\sigma \cdot c} \pbrac{i} = c \pbrac*{\sigma^{-1} \pbrac{i}}$.
  Then an element of $F \transport{n} \times \poscols{n}$ is a \emph{colored $F$-structure}, and a specific element $\pbrac{T, c} \in F \transport{n} \times \poscols{n}$ is said to have \emph{coloring} $c$.

  Fix a vector $\pi = \abrac*{\pi_{1}, \dots, \pi_{k}}$ of positive integers summing to $n$ and let $c_{\pi} \in \poscols{n}$ be the coloring where the first $\pi_{1}$ integers are the fiber of $1$, the next $\pi_{2}$ are the fiber of $2$, and so on.
  Then a \emph{partially-labeled $F$-structure with profile $\pi$} is an orbit of an $F$-structure with coloring $c_{\pi}$ under the action of $\symgp{n}$.
\end{definition}

\begin{example}
  Let $\specgraph$ denote the species of simple graphs.
  Fix the profile vector $\pi = \abrac{3, 1}$.
  \Cref{fig:plgraph} shows a colored $\specgraph \trans{4}$-structure (with the colors represented by node shapes) and its orbit under the action of $\symgp{4}$ (in schematic form, where the labels may be assigned freely).
\end{example}

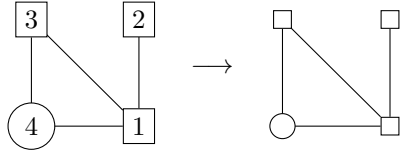
\begin{figure}[hb]
  \centering
  \begin{tikzpicture}[baseline=(current bounding box.base)]
    \node[style=snode] (A) at (-45:1) {1};
    \path [draw] (A) -- (45:1)  node [style=snode] (B) {2};
    \path [draw] (A) -- (135:1) node [style=snode] (C) {3};
    \path [draw] (C) -- (-135:1) node [style=cnode] (D) {4} -- (A);
  \end{tikzpicture}
  \quad $\longrightarrow$ \quad
  \begin{tikzpicture}[baseline=(current bounding box.base)]
    \node[style=snode] (A) at (-45:1) {};
    \path [draw] (A) -- (45:1)  node [style=snode] (B) {};
    \path [draw] (A) -- (135:1) node [style=snode] (C) {};
    \path [draw] (C) -- (-135:1) node [style=cnode] (D) {} -- (A);
  \end{tikzpicture}
  \caption{A colored simple graph with $4$ vertices and its associated partially-labeled graph with profile $\abrac{3, 1}$ (in schematic form)}
  \label{fig:plgraph}
\end{figure}

This notion of a partially-labeled $F$-structure refines the classical P\'{o}lya-theoretic notion of a `colored' $F$-structure.
In particular, if we can enumerate partially-labeled $F$-structures with all profiles, we can use this information to count the classical $k$-colored structures by summing over all profile vectors with $k$ parts.
In fact, the enumeration of partially-labeled $F$-structures can be completed with no more than the cycle index series $\ci{F}$, as is shown in \cite[eq.~4.3.23]{bll}.

\begin{theorem}[P\'{o}lya's theorem for species]
  \label{thm:polyaspec}
  Let $F$ be a combinatorial species with cycle index series $\ci{F}$ and fix a vector $\pi = \abrac{\pi_{1}, \pi_{2}, \dots, \pi_{k}}$ of positive integers.
  Let $\eta: \Lambda \to P$ be the map which expands each abstract symmetric function as a formal power series in variables $x_{i}$.
  Then the number of partially-labeled $F$-structures of profile $\pi$ is equal to the coefficient of $x_{\pi}$ in $\eta \pbrac{\ci{F}}$.
\end{theorem}

This notion of \enquote{partially-labeled} structures allows us to interpolate between labeled and unlabeled structures.
In particular, the notion of unlabeled $F$-structures of order $n$ may be recovered by taking the partially-labeled $F$-structures of profile $\pi = \abrac{n}$, while the labeled $F$-structures of order $n$ may be recovered by taking the partially-labeled $F$-structures of profile $\pi = \abrac{1, 1, \dots, 1}$.
This leads to a straightforward proof of two important enumerative results on species.

\begin{theorem}[\protect{\cite[\S 1.2, Thm.~8]{bll}}]
  \label{thm:specgf}
  Let $F$ be a combinatorial species.
  Denote by $F(x)$ the exponential generating function of labeled $F$-structures and by $\tilde{F}(x)$ the ordinary generating function of unlabeled $F$-structures.
  Then we have the following identities of formal power series:
  \begin{subequations}
    \begin{gather}
      F(x) = \civars{F}{x,0,0,\dots} \label{eq:ciegf} \\
      \tilde{F}(x) = \civars{F}{x, x^{2}, x^{3}, \dots} \label{eq:ciogf}
    \end{gather}
  \end{subequations}
\end{theorem}

\subsection{Species-theoretic enumeration of rooted binary leaf-multi-labeled trees}
As an example of the application of this theory, we now investigated the \enquote{rooted binary leaf-multi-labeled trees} of \cite{cejm:mltrees}.
To begin, we will consider the species $\specrblt$ consisting of rooted binary trees whose internal nodes are unlabeled.
Letting $\speclin[2]$ denote the species of lists of length $2$, we clearly have that
\begin{equation}
  \label{eq:rblt}
  \specrblt = \specsing[X] + \speclin[2] \pbrac*{\specrblt}.
\end{equation}
This allows for recursive calculation of the two-sort cycle index series of $\specrblt$.

In light of the application in \cite{cejm:mltrees}, we are interested in the enumeration of $\specrblt$-structures which are partially-labeled from a set of $k$ labels.
Let $\eta_{k}: \Lambda \to P$ denote the map which expands each symmetric variables in the family $\cbrac{x_{1}, \dots, x_{k}}$ of $k$ indeterminates and let $\lambda = \sbrac{\lambda_{1}, \lambda_{2}, \dots, \lambda_{i}} \vdash n$ be a partition with no more than $k$ parts.
Then, by \cref{thm:polyaspec}, the coefficient of $x_{\lambda}$ in $\eta_{k} \pbrac*{\ci{\specrblt}}$ is the number of $\specrblt$-structures with $n$ leaves with $\lambda_{1}$ of them labeled $1$, $\lambda_{2}$ labeled $2$, and so forth.
The total number of $k$-multi-labeled $\specrblt$-structures with $n$ vertices is then simply the sum of the coefficients of the degree-$n$ terms in $\eta_{k} \pbrac*{\ci{\specrblt}}$.

We can compute these numbers using the Sage code appearing in \cref{list:rblt}.
This code is shown configured to compute the number of rooted binary leaf-multi-labeled trees with $8$ leaves labeled from $\sbrac{4}$, which it finds to be $366680$ (in agreement with \cite[Table 1]{cejm:mltrees}).

\begin{lstlisting}[float=htbp,caption=Sage code to compute numbers of rooted binary leaf-multi-labeled trees, label=list:rblt, language=Python, texcl=true, xleftmargin=3em]
from sage.combinat.species.library import SingletonSpecies,LinearOrderSpecies

X = species.SingletonSpecies()
L2 = species.SetSpecies(size=2)

RootedBinaryLeafTrees = species.CombinatorialSpecies()
RootedBinaryLeafTrees.define(X + L2(RootedBinaryLeafTrees))

RBLT_sf = RootedBinaryLeafTrees.cycle_index_series().expand_as_sf(4)

print sum(RBLT_sf.coefficient(8).coefficients())
\end{lstlisting}

\section{$\Gamma$-species}
\label{sec:gspecies}
\subsection{Groups acting on species}
\label{sec:groupactspec}
Now let us consider, for a fixed species $F$, the case of a species isomorphism $\phi: F \to F$, which we hereafter call a \emph{species automorphism}.
Diagramatically, this is a choice of a \enquote{set automorphism} (i.e.~permutation) $\phi_{A}: A \to A$ for each $A \in \catname{FinSet}$ such that the diagram in \cref{fig:specautdiag} commutes for all $\sigma \in \symgp{A}$.

\begin{figure}[htbp]
  \centering
  \begin{tikzpicture}[auto,node distance = 2cm]
    \node (A1) {$A$};
    \node (A2) [right of = A1] {$A$};
    \node (FA1t) [above left of = A1] {$F \sbrac{A}$};
    \node (FA1b) [below left of = A1] {$F \sbrac{A}$};
    \node (FA2t) [above right of = A2] {$F \sbrac{A}$};
    \node (FA2b) [below right of = A2] {$F \sbrac{A}$};

    \draw [->] (A1) -- node {$\sigma$} (A2);
    \draw [->] (A1) -- node [swap] {$F$} (FA1t);
    \draw [->] (A1) -- node {$F$} (FA1b);
    \draw [->] (A2) -- node {$F$} (FA2t);
    \draw [->] (A2) -- node [swap] {$F$} (FA2b);
    \draw [->] (FA1t) -- node {$F \sbrac{\sigma}$} (FA2t);
    \draw [->] (FA1b) -- node [swap] {$F \sbrac{\sigma}$} (FA2b);
    \draw [->] (FA1t) -- node [swap] {$\phi_{A}$} (FA1b);
    \draw [->] (FA2t) -- node {$\phi_{A}$} (FA2b);
  \end{tikzpicture}
  \caption{Diagram which must commute if $\phi$ is a species automorphism}
  \label{fig:specautdiag}
\end{figure}

In other words, $\phi_{A}$ is just a permutation of $F \sbrac{A}$ which commutes with all the permutations $F \sbrac{\sigma}$.
This corresponds to the combinatorial notion of a \enquote{structural} or \enquote{label-independent} operation, such as taking the complement of a graph, permuting the colors of a colored graph, or cyclically rotating a finite sequence.

Many important problems in enumerative combinatorics arise when considering the classes of structures which are \enquote{equivalent} under the operation of such a structural operation (or, often, several such operations acting in concert).
In particular, if a group $\Gamma$ acts \enquote{structurally} (i.e. by structural operations) on a combinatorial class, the equivalence classes under $\Gamma$ are the \enquote{$\Gamma$-quotient structures}.

We can capture this idea efficiently in the language of species; we simply want to describe a group $\Gamma$ acting by species isomorphisms $F \to F$ for a fixed species $F$.
Since the collection $\Aut \pbrac{F}$ of all species automorphisms of $F$ already forms a group, we can achieve this classically by taking a specified homomorphism $\Gamma \to \Aut \pbrac{F}$.

Categorically, $\Gamma$ is simply a groupoid with a single object, so we can also achieve our association of $\Gamma$ with some of $F$'s automorphisms by constructing a functor sending $\Gamma$ to $F$ and each element $\gamma$ of $\Gamma$ to some automorphism $\gamma'$ of $F$ in a structure-preserving way.
In other words, we need a functor from $\Gamma$ to $\catname{Spec}$ whose object-image is $F$.

This leads to a very compact definition: for a group $\Gamma$, a \emph{$\Gamma$-species} is a functor $F: \Gamma \to \catname{Spec}$.
If $F$ is a $\Gamma$-species, its \emph{quotient} is the species $\quot{F}{\Gamma}$ defined as follows:
\begin{itemize}
\item
  For a given label set $A$, each element of $\quot{F}{\Gamma} \sbrac{A}$ is a \emph{set} of $F$-structures which form an orbit under the action of $\Gamma$.
\item
  For a given permutation $\sigma \in \sbrac{A}$, the transport $\quot{F}{\Gamma} \sbrac{\sigma}$ sends each $\Gamma$-orbit of $F$-structures labeled by $A$ to the orbit containing the images of the original structures under $\sigma$.
  (This is well-defined because the images of the morphisms of $\Gamma$ are natural isomorphisms of $F$ and thus commute with permutations.)
\end{itemize}

Just as with ordinary species, we can associate to each $\Gamma$-species a formal power series which encodes much of the relevant enumerative data.
This is the \emph{$\Gamma$-cycle index series}, which associates to each element $\gamma$ of $\Gamma$ a classical cycle index series.
The $\Gamma$-cycle index series of a $\Gamma$-species $F$ is given by
\begin{equation}
  \label{eq:gcidefperm}
  \gcielt{\Gamma}{F}{\gamma} = \sum_{n \geq 0} \frac{1}{n!} \sum_{\sigma \in \symgp{n}} \fix \pbrac{\gamma_{\intset{n}} \cdot F \sbrac{\sigma}} p_{1}^{\sigma_{1}} p_{2}^{\sigma_{2}} \dots
\end{equation}
where $\intset{n} = \cbrac{1, 2, \dots, n}$ is a canonical $n$-element set, $\gamma_{A}$ is the permutation $A \to A$ induced by $\gamma$, and $\gamma_{A} \cdot F \sbrac{\sigma}$ is the operation which first permutes the $F$-structures using $\sigma$ and then applies $\gamma_{A}$.
By functorality, $\fix \pbrac{\gamma_{\intset{n}} \cdot F \sbrac{\sigma}}$ is actually a class function on permutations $\sigma \in \symgp{n}$, so we can instead work at the level of conjugacy classes (indexed by partitions of $n$).
In this light, the $\Gamma$-cycle index of a $\Gamma$-species $F$ is given by
\begin{equation}
  \label{eq:gcidefpart}
  \gcielt{\Gamma}{F}{\gamma} = \sum_{n \geq 0} \sum_{\lambda \vdash n} \fix \pbrac{\gamma_{\intset{n}} \cdot F \sbrac{\lambda}} \frac{p_{1}^{\lambda_{1}} p_{2}^{\lambda_{2}} \dots}{z_{\lambda}}
\end{equation}
for $\fix F \sbrac{\lambda} = \fix F \sbrac{\sigma}$ for some choice of permutation $\sigma \in \symgp{n}$ of cycle type $\lambda$, for $\lambda_{i}$ the number of $i$ parts in $\lambda$, and for $z_{\lambda} = \prod_{i} i^{\lambda_{i}} \lambda_{i}!$ the number such that there are $n! / z_{\lambda}$ permutations of cycle type $\lambda$.
(Note that, in particular, for $e$ the identity element of the group $\Gamma$, we necessarily have $\gcielt{\Gamma}{F}{e} = \ci{F}$, the ordinary cycle index of the underlying actionless species $F$.)

The algebra of $\Gamma$-cycle indices is implemented by the \code{GroupCycleIndexSeries} class of Sage \cite{sage}.
We will demonstrate its use in \cref{sec:gspecenum}.

\subsection{$\Gamma$-species maps}
\label{s:gspecmap}
Continuing in the categorical theme, we now define an appropriate notion of \enquote{morphism} for the context of $\Gamma$-species.
Since a $\Gamma$-species is a functor, one reasonable approach is simply to say that a morphism of $\Gamma$-species $F$ and $G$ is a natural transformation $\phi: F \to G$.
However, since $\Gamma$-species are functors whose codomains are themselves functor categories, this requires some unpacking.
Additionally, this definition would in fact allow for the possibility of morphisms between the groups acting on the species, creating additional complexity for limited benefit (since we will generally only be interested in isomorphisms at this level).
Thus, we will take a more concrete approach to the definition.

Suppose $F$ and $G$ are $\Gamma$-species and let $\phi: F \to G$ be a species map of the underlying combinatorial species.
We wish to characterize the sense in which $\phi$ may be \enquote{compatible} with the $\Gamma$-actions on $F$ and $G$.
For any two label sets $A$ and $B$ with bijection $\sigma: A \to B$ and any element $\gamma \in \Gamma$, the fact that $F$ and $G$ are $\Gamma$-species implies the existence of several maps:
\begin{itemize}
\item
  bijections $F \sbrac{\sigma}: F \sbrac{A} \to F \sbrac{B}$ and $G \sbrac{\sigma}: G \sbrac{A} \to G \sbrac{B}$;

\item
  permutations $\gamma_{F \sbrac{A}}: F \sbrac{A} \to F \sbrac{A}$, $\gamma_{F \sbrac{B}}: F \sbrac{B} \to F \sbrac{B}$, $\gamma_{G \sbrac{A}}: G \sbrac{A} \to G \sbrac{A}$, and $\gamma_{G \sbrac{B}}: G \sbrac{B} \to G \sbrac{B}$;

\item
  and set maps $\phi_{A}: F \sbrac{A} \to G \sbrac{A}$ and $\phi_{B}: F \sbrac{B} \to G \sbrac{B}$.
\end{itemize}

We can construct a diagram which encodes the relationships among all these maps; this is shown in \cref{fig:gspecmapdiag}.
This diagram automatically has substantial commutativity: the inner and outer squares commute because $F$ and $G$ are species, and the top and botton squares commute because $\phi$ is a species morphism.
All that is required to make $\phi$ compatible with $\gamma$ is that the left and right squares commute as well.
This gives us our more concrete defintion of a $\Gamma$-species morphism.

\begin{definition}
  \label{def:gspecmap}
  Let $F$ and $G$ be $\Gamma$-species.
  Then a \emph{$\Gamma$-species map} $\phi: F \to G$ is a choice of a set map $\phi_{A}: F \sbrac{A} \to G \sbrac{A}$ for each finite set $A$ such that the diagram in \cref{fig:gspecmapdiag} commutes for every set bijection $\sigma: A \to B$.
  (Equivalently, $\phi$ is a natural transformation $F \to G$.)
  If every map $\phi_{L}$ is a bijection, $\phi$ is a \emph{$\Gamma$-species isomorphism}.
  If every map $\phi_{L}$ is an injection, $\phi$ is a \emph{$\Gamma$-species embedding}.
  If every map $\phi_{L}$ is a surjection, $\phi$ is a \emph{$\Gamma$-species covering}.

  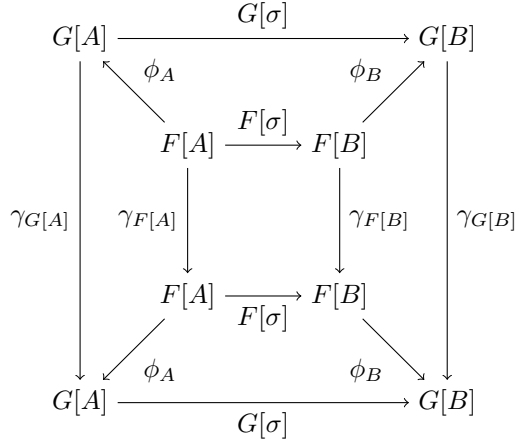
\begin{figure}[htbp]
    \centering
    \begin{tikzpicture}[auto,node distance = 2cm]
      \node (FAt) {$F \sbrac{A}$};
      \node (FAb) [below of = FAt] {$F \sbrac{A}$};
      \node (FBt) [right of = FAt] {$F \sbrac{B}$};
      \node (FBb) [below of = FBt] {$F \sbrac{B}$};

      \node (GAt) [above left of = FAt] {$G \sbrac{A}$};
      \node (GAb) [below left of = FAb] {$G \sbrac{A}$};
      \node (GBt) [above right of = FBt] {$G \sbrac{B}$};
      \node (GBb) [below right of = FBb] {$G \sbrac{B}$};

      \draw [->] (FAt) -- node {$F \sbrac{\sigma}$} (FBt);
      \draw [->] (FAb) -- node [swap] {$F \sbrac{\sigma}$} (FBb);
      \draw [->] (FAt) -- node [swap] {$\gamma_{F \sbrac{A}}$} (FAb);
      \draw [->] (FBt) -- node {$\gamma_{F \sbrac{B}}$} (FBb);

      \draw [->] (GAt) -- node {$G \sbrac{\sigma}$} (GBt);
      \draw [->] (GAb) -- node [swap] {$G \sbrac{\sigma}$} (GBb);
      \draw [->] (GAt) -- node [swap] {$\gamma_{G \sbrac{A}}$} (GAb);
      \draw [->] (GBt) -- node {$\gamma_{G \sbrac{B}}$} (GBb);

      \draw [->] (FAt) -- node [swap] {$\phi_{A}$} (GAt);
      \draw [->] (FAb) -- node {$\phi_{A}$} (GAb);
      \draw [->] (FBt) -- node {$\phi_{B}$} (GBt);
      \draw [->] (FBb) -- node [swap] {$\phi_{B}$} (GBb);
    \end{tikzpicture}
    \caption{Diagram which must commute if $\phi$ is a $\Gamma$-species map}
    \label{fig:gspecmapdiag}
  \end{figure}
\end{definition}

We note that he definitions of $\Gamma$-species isomorphism, embedding, and covering are simply the definitions of species isomorphism, embedding, and covering from \cite[Def.~1.1.4]{argthesis} combined with the compatibility condition.
When there exists a $\Gamma$-species isomorphism $\phi: F \to G$, we will often simply write $F = G$, omitting reference to the specific isomorphism.

With this notion of $\Gamma$-species morphism in hand, we note that the class of all $\Gamma$-species forms a category, which we denote $\catname[\Gamma]{Spec}$.

\subsection{P\'{o}lya theory for $\Gamma$-species}
\label{sec:polya.gspec}
We now revisit the core ideas of \cref{sec:polyaspec} in the context of $\Gamma$-species.
First, we present a useful generalization of Burnside's lemma, which appears (in an unweighted form) as \cite[eq.~A1.51]{bll}.

\begin{lemma}[Weighted generalized Burnside's lemma]
  \label{thm:burnside.gen}
  Let $G$ and $H$ be groups with commuting actions on a finite set $X$.
  Let $W: X \to \fieldname{A}$ be a \emph{weight function} from $X$ to a $\fieldname{Q}$-module $\fieldname{A}$ which is constant on $\pbrac*{G \times H}$-orbits.
  For any endofunction $f: X \to X$, let $\Fix_{X} (f) = \sum_{x \in X, f(x) = x} W(x)$ denote the sum of the weights of the fixed points of $f$.
  Then the sum of the weights of $H$-orbits fixed by a given element $g \in G$ is
  \begin{equation}
    \label{eq:burnside.gen}
    \Fix_{X / H} (g) = \frac{1}{\abs{H}} \sum_{h \in H} \Fix_{X} \pbrac*{g, h}.
  \end{equation}
\end{lemma}

Now let $F$ be a $\Gamma$-species and let $\gci{\Gamma}{F}$ be its $\Gamma$-cycle index series.
Partial labelings of $F$-structures are easily seen to be $\Gamma$-equivariant.
Thus, we can extend our P\'{o}lya theory for species to incorporate $\Gamma$-species.

\begin{theorem}[P\'{o}lya's theorem for $\Gamma$-species]
  \label{thm:polyagspec}
  Let $F$ be a $\Gamma$-species with $\Gamma$-cycle index series $\gci{\Gamma}{F}$ and fix a vector $\pi = \abrac{\pi_{1}, \dots, \pi_{k}}$ of positive integers.
  Let $\eta: \Lambda \to P$ be the map which expands each abstract symmetric function as a formal power series in variables $x_{i}$.
  Then the number of partially-labeled $F$-structures of profile $\pi$ which are fixed by the action of $\gamma \in \Gamma$ is equal to the coefficient of $x_{\pi}$ in $\eta \pbrac*{\gcielt{\Gamma}{F}{\gamma}}$.
\end{theorem}

\begin{proof}
  Following the notions introduced in \cref{def:plstruct}, we let $\symgp{n}$ act on $\poscols{n}$ and thus on $F \transport{n} \times \poscols{n}$.
  Let $\Gamma$ act trivially on $\poscols{n}$; then $\Gamma$ acts on $F \transport{n} \times \poscols{n}$ also, and the actions of $\Gamma$ and $\symgp{n}$ commute.

  Let $P$ denote the $\fieldname{Q}$-module of formal power series in the countably infinite family of variables $x_{1}, x_{2}, \dots$.
  Define a weight function $W: F \transport{n} \times \poscols{n} \to P$ by $W \pbrac*{T, c} = \prod_{i \in \intset{n}} x_{c(i)}$.
  It is clear that $W$ is constant on $\pbrac*{\Gamma \times \symgp{n}}$-orbits.

  Fix some $\gamma \in \Gamma$.
  By \cref{thm:burnside.gen}, the sum of the weights of all the $\symgp{n}$-orbits fixed by $\gamma$ in $F \transport{n} \times \poscols{n}$ is given by
  \begin{equation}
    \label{eq:burnside.gen.gspec.outer}
    \Fix_{\pbrac*{F \transport{n} \times \poscols{n}} / \symgp{n}} (\gamma) = \frac{1}{n!} \sum_{\sigma \in \symgp{n}} \Fix_{F \transport{n} \times \poscols{n}} \pbrac*{\gamma, \sigma}.
  \end{equation}
  For a given $\sigma \in \symgp{n}$, it is clear that a pair $\pbrac*{T, c} \in F \transport{n} \times \poscols{n}$ is fixed by $\pbrac*{\gamma, \sigma}$ if and only if the $F$-structure $T$ and the coloring $c$ are fixed separately.

  For a coloring $c \in \poscols{n}$ to be fixed by $\pbrac*{\gamma, \sigma}$, it must be fixed by $\sigma$.
  This occurs if and only if $c$ is constant on each orbit of $\sigma$ in $\intset{n}$, so the sum of all the weights of the colorings $c$ fixed by $\sigma$ is exactly $\eta \pbrac*{p_{\sigma}}$.
  Let $\fix \pbrac*{\gamma, \sigma}$ denote the number of $F \transport{n}$-structures fixed by $\pbrac*{\gamma, \sigma}$.
  Then the sum of the weights of the fixed points of $\pbrac*{\gamma, \sigma}$ in $F \transport{n} \times \poscols{n}$ is given by
  \begin{equation}
    \label{eq:burnside.gen.gspec.inner}
    \Fix_{F \transport{n} \times \poscols{n}} \pbrac*{\gamma, \sigma} = \fix \pbrac*{\gamma, \sigma} \eta \pbrac*{p_{\sigma}}.
  \end{equation}

  Combining \cref{eq:burnside.gen.gspec.outer,eq:burnside.gen.gspec.inner}, we have
  \begin{equation}
    \label{eq:burnside.gen.gspec}
    \Fix_{\pbrac*{F \transport{n} \times \poscols{n}} / \symgp{n}} (\gamma) = \frac{1}{n!} \sum_{\sigma \in \symgp{n}} \fix \pbrac*{\gamma, \sigma} \eta \pbrac*{p_{\sigma}}.
  \end{equation}
  The desired result follows from summing over all $n$ in \cref{eq:burnside.gen.gspec}.
\end{proof}

We note that \cref{thm:polyaspec} is an immediate consequence of \cref{thm:polyagspec}, providing a proof which does not require \cref{thm:cisrel}.
It is also natural to extend \cref{thm:specgf} to the $\Gamma$-species context.

\begin{theorem}
  \label{thm:gspecgf}
  Let $F$ be a $\Gamma$-species and let $\gamma \in \Gamma$.
  Denote by $F_{\gamma} (x)$ the exponential generating function of labeled $F$-structures fixed by $\gamma$ and by $\tilde{F}_{\gamma} (x)$ the ordinary generating function of unlabeled $F$-structures fixed by $\gamma$.
  Then we have the following identities of formal power series:
  \begin{subequations}
    \begin{gather}
      F_{\gamma} (x) = \gcieltvars{\Gamma}{F}{\gamma}{x, 0, 0, \dots} \label{eq:gciegf} \\
      \tilde{F}_{\gamma} (x) = \gcieltvars{\Gamma}{F}{\gamma}{x, x^{2}, x^{3}, \dots} \label{eq:gciogf}
    \end{gather}
  \end{subequations}
\end{theorem}

\begin{proof}
  There is a natural bijection between labeled $\gamma$-fixed $F$-structures with $n$ vertices and partially-labeled $\gamma$-fixed $F$-structures of profile $\abrac{1, 1, \dots, 1}$.
  Thus, the number of labeled $\gamma$-fixed $F$-structures with $n$ vertices is the coefficient of $x_{1} x_{2} \dots x_{n}$ in $\eta \pbrac*{\gcielt{\Gamma}{F}{\gamma}}$.
  Such a term can only appear in $\eta \pbrac*{p_{\sigma}}$ if $\sigma$ is the identity permutation, so this is equal to the coefficient of $p_{1}^{n}$.
  \Cref{eq:gciegf} follows.

  Similarly, there is a natural bijection between unlabeled $\gamma$-fixed $F$-structures with $n$ vertices and partially-labeled $\gamma$-fixed $F$-structures of profile $\abrac{n}$.
  Thus, the number of unlabeled $\gamma$-fixed $F$-structures with $n$ vertices is the coefficient of $x_{1}^{n}$ in $\eta \pbrac*{\gcielt{\Gamma}{F}{\gamma}}$.
  \emph{Every} symmetric function $p_{\sigma}$ contributes such a term, so this is the sum of \emph{all} the coefficients on terms of degree $n$ in $\gcielt{\Gamma}{F}{\gamma}$.
  \Cref{eq:gciogf} follows.
\end{proof}

\begin{observation}
  Crucially, each of $F_{\gamma} (x)$ and $\tilde{F}_{\gamma} (x)$ counts structures which are fixed by $\gamma$ \emph{with respect to their partial labelings}.
  Thus, $F_{\gamma} (x)$ counts only those labeled $F$-structures which are fixed \emph{as labeled structures}, while $\tilde{F}_{\gamma} (x)$ counts unlabeled $F$-structures which are fixed \emph{as unlabeled structures}.
\end{observation}

\section{Algebra of $\Gamma$-species}
\label{sec:gspeciesalg}

\subsection{Species quotients}
\label{sec:specquot}
Classical species-theoretic enumeration uses the cycle index series to \enquote{keep the books} on the actions of symmetric groups on the labels of combinatorial structures, then apply Burnside's lemma to take quotients at an appropriate time.
$\Gamma$-species theory extends this practice, using the $\Gamma$-cycle index series to analogously \enquote{keep the books} on these actions and some structural group action \emph{simultaneously}, then apply Burnside's lemma to one or both as appropriate.
\begin{lemma}[{\cite[Thm.~1.5.11]{argthesis}}]
  \label{thm:gciquot}
  Let $F$ be a $\Gamma$-species.
  Then we have
  \begin{equation}
    \label{eq:ciquot}
    \ci{\quot{F}{\Gamma}} = \frac{1}{\abs{\Gamma}} \sum_{\gamma \in \Gamma} \gcielt{\Gamma}{F}{\gamma}.
  \end{equation}
  where $\ci{\quot{F}{\Gamma}}$ is the classical cycle index of the quotient species $\quot{F}{\Gamma}$ and $\gci{\Gamma}{F}$ is the $\Gamma$-cycle index of $F$.
\end{lemma}

(The proof of \cref{thm:gciquot} is essentially a careful application of Burnside's Lemma and appears in full in \cite{argthesis}.)

It remains, of course, to develop an algebraic theory facilitating the computation of $\Gamma$-cycle indices analogous to that available for classical cycle indices.
Fortunately, this is not difficult.
Each of the standard species operators $+$, $-$, $\cdot$, $\circ$, $\square$, $\pointed{}$, and $'$ has a natural analogue for $\Gamma$-species, corresponding to a suitable operation on $\Gamma$-cycle indices.

With the exception of $\circ$ and $\square$, the definitions of these $\Gamma$-species operations and their associated $\Gamma$-cycle index operations are completely natural, so we omit them here.
However, the two composition operators are more subtle.

\subsection{Plethystic composition of $\Gamma$-species}
\label{sec:plethgspec}

In the classical setting, if $F$ and $G$ are species, an $\pbrac{F \circ G}$-structure is an \enquote{$F$-structure of $G$-structures}.
Extending this to the $\Gamma$-species setting does not require changing our understanding of the structures; the difficulty is in making sense of how an element $\gamma$ of $\Gamma$ should act on such a structure.

Consider a schematic $\pbrac{F \circ G}$-structure, as in \cref{fig:fgschem}.

\begin{figure}[htbp]
  \centering
  \begin{tikzpicture}
    \node [invisinode] (A) at (0,0) {$F$};

    \node [invisinode] (B) at (-2,-1) {$G$};
    \node [cfnode] (b1) at (-2.5,-2) {};
    \node [cfnode] (b2) at (-1.5,-2) {};

    \node [invisinode] (C) at (0,-1) {$G$};
    \node [cfnode] (c1) at (-.5,-2) {};
    \node [cfnode] (c2) at (.5,-2) {};

    \node [invisinode] (D) at (2,-1) {$G$};
    \node [cfnode] (d1) at (1.5,-2) {};
    \node [cfnode] (d2) at (2.5,-2) {};

    \draw
    (A) -- (B) (B) -- (b1) (B) -- (b2)
    (A) -- (C) (C) -- (c1) (C) -- (c2)
    (A) -- (D) (D) -- (d1) (D) -- (d2);

    \node [dashbox, fit=(B) (C) (D)] {};
    \node [dashbox, fit=(b1) (b2)] {};
    \node [dashbox, fit=(c1) (c2)] {};
    \node [dashbox, fit=(d1) (d2)] {};
  \end{tikzpicture}
  \caption{A schematic of an $\pbrac{F \circ G}$-structure}
  \label{fig:fgschem}
\end{figure}
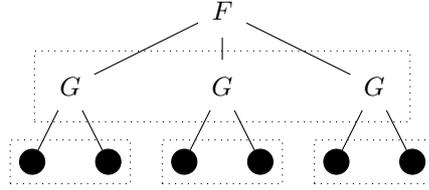

The \enquote{parent $F$-structure} and each of the \enquote{descendant $G$-structures} is modified in some way by the action of a particular element $\gamma \in \Gamma$.
To obtain a action of $\gamma \in \Gamma$ on the aggregate $\pbrac{F \circ G}$-structure, we can simply apply $\gamma$ to each of the descendant $G$-structures independently and then to the parent $F$-structure, as illustrated in \cref{fig:fgschemact}.

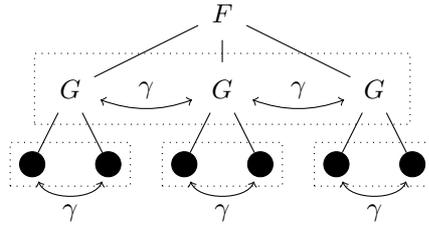
\begin{figure}[htbp]
  \centering
  \begin{tikzpicture}
    \node [invisinode] (A) at (0,0) {$F$};

    \node [invisinode] (B) at (-2,-1) {$G$};
    \node [cfnode] (b1) at (-2.5,-2) {};
    \node [cfnode] (b2) at (-1.5,-2) {};

    \node [invisinode] (C) at (0,-1) {$G$};
    \node [cfnode] (c1) at (-.5,-2) {};
    \node [cfnode] (c2) at (.5,-2) {};

    \node [invisinode] (D) at (2,-1) {$G$};
    \node [cfnode] (d1) at (1.5,-2) {};
    \node [cfnode] (d2) at (2.5,-2) {};

    \draw
    (A) -- (B) (B) -- (b1) (B) -- (b2)
    (A) -- (C) (C) -- (c1) (C) -- (c2)
    (A) -- (D) (D) -- (d1) (D) -- (d2);

    \node [dashbox, fit=(B) (C) (D)] {};
    \node [dashbox, fit=(b1) (b2)] {};
    \node [dashbox, fit=(c1) (c2)] {};
    \node [dashbox, fit=(d1) (d2)] {};

    \path [<->, every edge/.append style={bend right=70, shorten <= 2pt, shorten >= 2pt}]
    (b1) edge node [below] {$\gamma$} (b2)
    (c1) edge node [below] {$\gamma$} (c2)
    (d1) edge node [below] {$\gamma$} (d2);

    \path [<->, every edge/.append style={bend right=20, shorten <= 2pt, shorten >= 2pt}]
    (B) edge node [above] {$\gamma$} (C)
    (C) edge node [above] {$\gamma$} (D);
  \end{tikzpicture}
  \caption{A schematic of an $\pbrac{F \circ G}$-structure with an action of $\gamma \in \Gamma$}
  \label{fig:fgschemact}
\end{figure}

It is shown in \cite[\S 4]{hend} that there is a corresponding operation $\circ$, the \enquote{$\Gamma$-cycle index plethysm}, given by
\begin{multline}
  \label{eq:gcipleth}
  \gci{\Gamma}{F \circ G} = \pbrac*{\gci{\Gamma}{F} \circ \gci{\Gamma}{G}} \pbrac{\gamma} \sbrac{p_{1}, p_{2}, p_{3}, \dots} = \\
  \gcielt{\Gamma}{F}{\gamma} \sbrac*{ \gcielt{\Gamma}{G}{\gamma} \sbrac{p_{1}, p_{2}, p_{3}, \dots}, \gcielt{\Gamma}{G}{\gamma^{2}} \sbrac{p_{2}, p_{4}, p_{6}, \dots}, \gcielt{\Gamma}{G}{\gamma^{3}} \sbrac{p_{3}, p_{6}, p_{9}, \dots}, \dots }.
\end{multline}

It is crucial to note a subtle point in \cref{eq:gcipleth}: the $\gci{\Gamma}{G}$ terms are evaluated at \emph{different powers of $\gamma$}.
To see why, recall that the coefficients of $\gcielt{\Gamma}{F}{\gamma}$ count $F$-structures that are \emph{fixed} by the combined action of $\gamma$ and some $\sigma \in \symgp{n}$.
In \cref{eq:gcipleth}, each evaluation of $\gci{\Gamma}{G}$ which is substituted for $p_{i}$ corresponds to a $G$-structure which is set into a $i$-cycle of the parent $F$-structure; if the overall $\pbrac{F \circ G}$-structure is to be fixed by $\gamma$, this descendant $G$-structure must be returned to itself after it moves around its $i$-cycle, which results in an application of $\gamma^{i}$.
Thus, the appropriate $\Gamma$-cycle index to substitute is $\gcielt{\Gamma}{G}{\gamma^{i}}$.

The operation $\circ$ is implemented by the \code{composition()} method of the \code{GroupCycleIndexSeries} class in Sage \cite{sage}.
We will demonstrate its use in \cref{sec:gspecenum}.

\subsection{Functorial composition of $\Gamma$-species}
\label{sec:funccompgspec}
Since a combinatorial species is a functor $\catname{FinSet} \to \catname{FinBij}$ and thus can be lifted to a functor $\catname{FinSet} \to \catname{FinSet}$, it is at least algebraically meaningful to consider the composition \emph{as functors} of two species.
This operation yields the \enquote{fuctorial composition} $F \funccomp G$.
An $\pbrac{F \funccomp G}$-structure on a label set $A$ is an $F$-structure on the set $G \sbrac{A}$ of $G$-structures labeled by $A$.
Although this operation is not as combinatorially natural as the plethystic composition $\circ$, it nevertheless is useful for certain constructions; for example, letting $\specsubs$ denote the species of subsets (i.e. $\specsubs = \specset \cdot \specset$) and $\specgraph$ the species of simple graphs, we have $\specgraph = \specsubs \funccomp \specsubs[2]$.

Since a $\Gamma$-species is formally a functor $\Gamma \to \catname{Spec}$, it is not meaningful to compose two $\Gamma$-species as functors.
However, if $F$ and $G$ are $\Gamma$-species, we can consider the functorial composition $F \funccomp G$ of their underlying classical species.
Each $\gamma \in \Gamma$ induces a permutation on the set $G \sbrac{A}$ of $G$-structures which commutes with the action of $\symgp{A}$.

Therefore, we can obtain a structural action of $\Gamma$ on $F \funccomp G$ in the following manner.
Consider a structure $s \in \pbrac{F \funccomp G} \sbrac{A}$ and fix an element $\gamma \in \Gamma$.
$s$ consists of an $F$-structure whose labels are all the $G$-structures in $G \sbrac{A}$.
Replace each $G$-structure with its image under $\gamma$, then apply $\gamma$ to the parent $F$-structure.
The result is a new $\pbrac{F \funccomp G}$-structure over $A$, since $\gamma$ must act by a bijection on $G \sbrac{A}$ and $F \sbrac{G \sbrac{A}}$.
That this action commutes with label permutations is clear.

Therefore, $F \funccomp G$ is in fact a $\Gamma$-species.
In light of the relationship to the functorial composition of classical species, we dub this the \emph{functorial composition of $F$ and $G$}, although (as previously noted) the composition of $F$ and $G$ as functors is not in fact well-defined.

It remains to compute the $\Gamma$-cycle index of the functorial composition of two $\Gamma$-species.
By definition,
\[ \gcielt{\Gamma}{F \funccomp G}{\gamma} = \sum_{n \geq 0} \frac{1}{n!} \sum_{\sigma \in \symgp{n}} \fix \pbrac*{ \gamma \cdot F \sbrac*{ \gamma \cdot G \sbrac*{\sigma} } } p_{\sigma}, \]
so we need only compute the values $\fix \pbrac*{ \gamma \cdot F \sbrac*{ \gamma \cdot G \sbrac*{\sigma} } }$ for each $\gamma \in \Gamma$ and $\sigma \in \symgp{n}$.
Since $\gamma \cdot G \sbrac*{\sigma}$ is a permutation on $G \sbrac*{n}$, this value already occurs as a coefficient in $\gcielt{\Gamma}{F}{\gamma}$.
We therefore take the following definition.

\begin{definition}
  \label{def:gcifunccomp}
  Let $F$ and $G$ be $\Gamma$-species.
  The \emph{functorial composite} $\gci{\Gamma}{F} \funccomp \gci{\Gamma}{G}$ of their $\Gamma$-cycle indices is the $\Gamma$-cycle index given by
  \begin{equation}
    \label{eq:gcifunccomp}
    \pbrac*{\gci{\Gamma}{F} \funccomp \gci{\Gamma}{G}} \pbrac{\gamma} = \sum_{n \geq 0} \frac{1}{n!} \sum_{\sigma \in \symgp{n}} \fix \pbrac*{\gamma \cdot F \sbrac*{\gamma \cdot G \sbrac*{\sigma}}} p_{\sigma}.
  \end{equation}
\end{definition}

That this corresponds to $\Gamma$-species functorial composition follows immediately.

\begin{theorem}
  \label{thm:gcifunccomp}
  Let $F$ and $G$ be $\Gamma$-species.
  The $\Gamma$-cycle index of their functorial composition is given by
  \begin{equation}
    \label{eq:gcifunccompthm}
    \gci{\Gamma}{F \funccomp G} = \gci{\Gamma}{F} \funccomp \gci{\Gamma}{G}.
  \end{equation}
\end{theorem}

It remains only to find a formula for the cycle type of the permutation $\gamma \cdot G \sbrac*{\sigma}$ on the set $G \intset{n}$.

\begin{lemma}
  \label{lem:funccompcyctype}
  Let $G$ be a $\Gamma$-species and fix $\gamma \in \Gamma$, $\sigma \in \symgp{n}$, and $k \geq 1$.
  The number of cycles of length $k$ in $\gamma \cdot G \sbrac{\sigma}$ as a permutation of $G \intset{n}$ is then given by
  \begin{equation}
    \label{eq:funccompcyctype}
    \pbrac*{\gamma \cdot G \sbrac{\sigma}}_{k} = \frac{1}{k} \sum_{d \mid k} \mu \pbrac*{\frac{k}{d}} \fix \pbrac*{\gamma^{d} \cdot G \sbrac*{\sigma^{d}}},
  \end{equation}
  where $\mu$ is the integer M\"{o}bius function.
\end{lemma}

The proof is similar to that of \cite[\S 2.2, Prop.~3]{bll}.
We present it here in full for completeness.

\begin{proof}
  We clearly have that
  \begin{equation}
    \label{eq:powerfixed}
    \fix \pbrac*{\pbrac*{\gamma \cdot G \sbrac*{\sigma}}^{k}} = \sum_{d \mid k} d \cdot \pbrac*{\gamma \cdot G \sbrac*{\sigma}}_{d}.
  \end{equation}
  Applying M\"{o}bius inversion to \cref{eq:powerfixed}, we obtain that
  \begin{equation}
    \label{eq:funccompminv}
    \pbrac*{\gamma \cdot G \sbrac*{\sigma}}_{k} = \frac{1}{k} \sum_{d \mid k} \mu \pbrac*{\frac{k}{d}} \fix \pbrac*{\pbrac*{\gamma \cdot G \sbrac*{\sigma}}^{d}}.
  \end{equation}
  Since $\gamma$ commutes with $G \sbrac*{\sigma}$, we can distribute the power of $d$ to the two rightmost terms in \cref{eq:funccompminv}; furthermore, by functorality, $G \sbrac*{\sigma}^{d} = G \sbrac*{\sigma^{d}}$.
  \Cref{eq:funccompcyctype} follows immediately.
\end{proof}

Thus, the cycle type of the permutation $\gamma \cdot G \sbrac{\sigma}$ may be computed using only the values of $\fix \pbrac*{\gamma' \cdot \sbrac*{\sigma'}}$ (allowing $\gamma'$ to range over $\Gamma$ and $\sigma'$ to range over $\symgp{n}$).
This justifies \cref{def:gcifunccomp}, since it implies that $\gci{\Gamma}{F} \funccomp \gci{\Gamma}{G}$ may be computed using only information taken from the coefficients of the $\Gamma$-cycle indices.

This operation is implemented by the \code{functorial\string_composition()} method of the \code{GroupCycleIndexSeries} class in Sage \cite{sage}.
We will demonstrate its use in \cref{sec:gspecenum}.

\section{Multisort $\Gamma$-species}
\label{sec:multisort}
Let $\catname[k]{FinSet}$ denote the category of finite $k$-sort sets (i.e. $k$-tuples of finite sets) whose morphisms are $k$-sort set maps.
A \emph{$k$-sort species} $F$ is then a functor $F: \catname[k]{FinSet} \to \catname{FinBij}$.
(Note that a $1$-sort species is simply a classical combinatorial species.)
$k$-sort species are useful for studying the combinatorics of structures which carry labels on several different \enquote{parts}; a natural example is the $2$-sort species of graphs with one sort of label on vertices and the other on edges.
(The use of $2$-sort set maps corresponds to the combinatorial fact that edge and vertex labels cannot be shuffled with each other.)

The algebraic theories of generating functions and cycle index series and the combinatorial calculus of species may all be extended naturally to the $k$-sort case for each $k$.
This is discussed at length in \cite[\S 2.4]{bll}.
We record here for future reference that the \emph{$k$-sort cycle index} of a $k$-sort species $F$ is given by
\begin{multline}
  \label{eq:kcidef}
  \ci{F} \pbrac*{p_{1, 1}, p_{1, 2}, \dots; p_{2, 1}, p_{2, 2}, \dots; \dots; p_{k, 1}, p_{k, 2}, \dots} = \\
  \sum_{n_{1}, n_{2}, \dots, n_{k} \geq 0} \frac{1}{n_{1}! n_{2}! \dots n_{k}!} \sum_{\sigma_{i} \in \symgp{n_{i}}} \abs*{\fix F \sbrac*{\sigma_{1}, \sigma_{2}, \dots, \sigma_{k}}} \, p_{1, 1}^{\sigma_{1,1}} p_{1, 2}^{\sigma_{1, 2}} \dots p_{k, 1}^{\sigma_{k, 1}} p_{k, 2}^{\sigma_{k, 2}} \dots
\end{multline}
where $p_{i, j}$ are a two-parameter infinite family of indeterminates and $\sigma_{i, j}$ is the number of $j$-cycles of $\sigma_{i}$.

Since a $k$-sort species $F$ admits a group $\Aut \pbrac{F}$ of automorphisms, we can translate the notion of a $\Gamma$-species to the $k$-sort context easily.
Specifically, a \emph{$k$-sort $\Gamma$-species} is a functor $F: \Gamma \to \catname[k]{Spec}$.
(As expected, a $1$-sort $\Gamma$-species is a classical $\Gamma$-species.)
Then the \emph{$k$-sort $\Gamma$-cycle index} of $F$ is given by
\begin{multline}
  \label{eqw:kgcidef}
  \gcielt{\Gamma}{F}{\gamma} \pbrac*{p_{1, 1}, p_{1, 2}, \dots; p_{2, 1}, p_{2, 2}, \dots; \dots; p_{k, 1}, p_{k, 2}, \dots} = \\
  \sum_{n_{1}, n_{2}, \dots, n_{k} \geq 0} \frac{1}{n_{1}! n_{2}! \dots n_{k}!} \sum_{\sigma_{i} \in \symgp{n_{i}}} \abs*{\fix \gamma_{\sbrac{n_{1}}, \dots, \sbrac{n_{k}}} \cdot F \sbrac*{\sigma_{1}, \sigma_{2}, \dots, \sigma_{k}}} \, p_{1, 1}^{\sigma_{1,1}} p_{1, 2}^{\sigma_{1, 2}} \dots p_{k, 1}^{\sigma_{k, 1}} p_{k, 2}^{\sigma_{k, 2}} \dots
\end{multline}
where $\gamma_{A_{1}, \dots, A_{k}}$ is the $k$-sort permutation $\pbrac*{A_{1}, \dots, A_{k}} \to \pbrac*{A_{1}, \dots, A_{k}}$ induced by $\gamma$, $p_{i, j}$ are a two-parameter infinite family of indeterminates, and $\sigma_{i, j}$ is the number of $j$-cycles of $\sigma_{i}$.

As always, the $k$-sort $\Gamma$-cycle index is compatible with the appropriate operations $+$ and $\cdot$ on $k$-sort $\Gamma$-species.
In addition, it is compatible with a suitable notion of \enquote{sorted substitution} which involves specifying a species to substitute for each sort of labels.

\section{Virtual $\Gamma$-species}
\label{s:virtgspec}
The theory of virtual species (developed by Yeh in \cite{yeh:virtspec}) elegantly resolves several algebraic problems in the theory of combinatorial species; in particular, it allows for subtraction of arbitrary species and the computation of compositional inverses of many species.
The key idea is simply to complete the semiring of combinatorial species with respect to the operations of species sum and species product.
Specifically, taking any two combinatorial species $F$ and $G$, we define their \emph{difference} $F - G$ to be the equivalence class of all pairs of species $\pbrac{A, B}$ of combinatorial species satisfying $F + B = G + A$ by species isomorphism.

This definition satisfies many desirable properties; perhaps most importantly, if $H = F + G$, then $H - F = G$ as an isomorphism of virtual species, and $F - F = 0$ for any virtual species $F$.

To extend this notion to the context of $\Gamma$-species is merely a matter of definition.
First, we define the relation which forms the classes.

\begin{definition}
  \label{def:gspecequiv}
  Fix a group $\Gamma$ and let $F$, $G$, $H$, and $K$ be $\Gamma$-species.
  We write $\pbrac{F, G} \sim \pbrac{H, K}$ if $F + K = G + H$ as an isomorphism of $\Gamma$-species in the sense of \cref{def:gspecmap}.
\end{definition}

It is straightforward to show that this relation $\sim$ is an equivalence.
Thus, we can use it as the basis for a definition of virtual $\Gamma$-species.

\begin{definition}
  \label{def:virtgspec}
  Fix a group $\Gamma$ and let $\catname[\Gamma]{Spec}$ denote the category of $\Gamma$-species.
  Then a \emph{virtual $\Gamma$-species} is an element of $\quot{\catname[\Gamma]{Spec} \times \catname[\Gamma]{Spec}}{\sim}$, where $\sim$ is the equivalence relation defined in \cref{def:gspecequiv}.
  If $F$ and $G$ are $\Gamma$-species, their \emph{difference} is the virtual $\Gamma$-species $\pbrac{F, G}$, frequently denoted $F - G$.
\end{definition}

We note that the elementary species $0$ and $1$ each admit a single (trivial) $\Gamma$-action for any group $\Gamma$.
These are the additive and multiplicative identities of the ring of virtual $\Gamma$-species.

We also note that any virtual $\Gamma$-species $\Phi = F - G$ has a $\Gamma$-cycle index series given by $\gci{\Gamma}{\Phi} = \gci{\Gamma}{F} - \gci{\Gamma}{G}$.
The fact that $\sim$ is an equivalence relation implies that we may choose any representative pair $\pbrac{F, G}$ for $\Phi$ and obtain the same $\Gamma$-cycle index series, so this is well-defined.

The remainder of Yeh's theory of virtual species extends automatically to the $\Gamma$-species context, so we will not develop it explicitly here.

\section{A library of elementary $\Gamma$-species}
\label{sec:gspeclib}
To illustrate the use of $\Gamma$-species, we will now compute explicitly the $\Gamma$-cycle indices of several important examples.

\subsection{Trivial actions}
\label{sec:trivact}
Any (virtual) species $F$ may be equipped with a trivial action by any group $\Gamma$ (that is, an action where every element of $\Gamma$ acts as the identity map on $F$-structures).
The $\Gamma$-cycle index of the (virtual) $\Gamma$-species $F$ obtained in this way is given by
\begin{equation}
  \label{eq:trivact}
  \gcielt{\Gamma}{F}{\gamma} = \ci{F}.
\end{equation}

\subsection{Linear and cyclic orders with reversal}
\label{sec:ordrev}
Let $\speclin$ denote the species of linear orders and $\speccyc$ denote the species of cyclic orders.
Each of these admits a natural $\symgp{2}$-action which sends each ordering to its reversal, and so we also have associated $\symgp{2}$-species $\speclin$ and $\speccyc$.

\begin{theorem}
  \label{thm:speclin}
  The $\symgp{2}$-cycle index series of the species $\speclin$ of linear orderings with the order-reversing action is given by
  \begin{subequations}
    \label{eq:speclin}
    \begin{gather}
      \gcielt{\symgp{2}}{\speclin}{e} = \ci{\speclin} = \frac{1}{1 - p_{1}} = 1 + p_{1} + p_{1}^{2} + p_{1}^{3} + \dots \label{eq:specline} \\
      \gcielt{\symgp{2}}{\speclin}{\tau} = \sum_{k = 1}^{\infty} p_{2}^{k} + p_{2}^{k} p_{1} \label{eq:speclint}
    \end{gather}
  \end{subequations}
  where $e$ denotes the identity element of $\symgp{2}$ and $\tau$ denotes the non-identity element.
\end{theorem}

\begin{proof}
  \Cref{eq:specline} appears with proof as \cite[eq.~1.2.15]{bll}.

  In light of \cref{eq:gcidefpart}, to compute $\gcielt{\symgp{2}}{\speclin}{\tau}$, we need to compute for each $n \geq 0$ and each $\lambda \vdash n$ the number of linear orders which are fixed by the combined action of $\tau$ and a permutation of cycle type $\lambda$.
  Clearly this is $0$ unless $\lambda$ is composed entirely of $2$'s and possibly a single $1$, as illustrated in \cref{fig:speclint}.
  If it does have this form, a permutation $\sigma$ of cycle type $\lambda$ will act by reversing the order of the lists $L$ which are constructed by the following process:
  \begin{enumerate}
  \item
    Choose an ordering on the $\floor*{\sfrac{n}{2}}$ $2$-cycles of $\sigma$.

  \item
    For each $2$-cycle, choose one of the two possible orderings of the pair of elements in that cycle.

  \item
    If applicable, place the element in the $1$-cycle in the center.
  \end{enumerate}
  Thus, there are $2^{\floor*{\sfrac{n}{2}}} \cdot \floor*{\sfrac{n}{2}}!$ lists whose order is reversed by the action of this $\sigma$.
  But, for a partition $\lambda$ of this form, we also have that $z_{\lambda} = 2^{\floor*{\sfrac{n}{2}}} \cdot \floor*{\sfrac{n}{2}}!$, and so the contribution to the cycle index series term $\gcielt{\symgp{2}}{\speclin}{\tau}$ is simply $1 p_{\lambda}$.
  \Cref{eq:speclint} follows.
\end{proof}

\begin{figure}[ht]
  \centering
  \begin{tikzpicture}
    \node[style=cnode] (a) {$a$};
    \node[style=cnode, right=of a] (b) {$b$};
    \node[right=of b] (ldots) {$\dots$};
    \node[style=cnode, right=of ldots] (c) {$c$};
    \node[right=of c] (rdots) {$\dots$};
    \node[style=cnode, right=of rdots] (d) {$d$};
    \node[style=cnode, right=of d] (e) {$e$};

    \path[style=diredge]
    (a) edge (b)
    (b) edge (ldots)
    (ldots) edge (c)
    (c) edge (rdots)
    (rdots) edge (d)
    (d) edge (e);

    \path[<->, dashed, bend left]
    (a) edge (e)
    (b) edge (d);
  \end{tikzpicture}
  \caption{Schematic of a permutation (dashed arrows) which reverses a linear order (solid arrows)}
  \label{fig:speclint}
\end{figure}
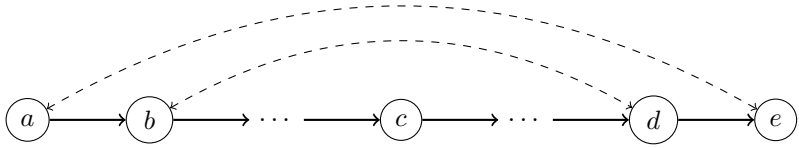

The $\symgp{2}$-cycle index $\gci{\symgp{2}}{\speclin}$ of $\speclin$ is available in Sage \cite{sage}:
\begin{description}
\item[module:]
  \code{sage.combinat.species.group\string_cycle\string_index\string_series\string_library}

\item[class:]
  \code{LinearOrderWithReversalGroupCycleIndex()}
\end{description}

\begin{theorem}
  \label{thm:speccyc}
  The $\symgp{2}$-cycle index series of the species $\speccyc$ of cyclic orderings with the order-reversing action is given by
  \begin{subequations}
    \label{eq:speccyc}
    \begin{gather}
      \gcielt{\symgp{2}}{\speccyc}{e} = \ci{\speccyc} = -\sum_{k = 1}^{\infty} \frac{\phi \pbrac{k}}{k} \ln \frac{1}{1 - p_{1}} \label{eq:speccyce} \\
      \gcielt{\symgp{2}}{\speccyc}{\tau} = \sum_{k = 1}^{\infty} \frac{1}{2} \pbrac*{p_{2}^{k} + p_{2}^{k - 1} p_{1}^{2}} + p_{2}^{k} p_{1} \label{eq:speccyct}
    \end{gather}
    where $\phi$ is the Euler $\phi$-function, $e$ is the identity element of $\symgp{2}$, and $\tau$ is the non-identity element of $\symgp{2}$.
  \end{subequations}
\end{theorem}

\begin{proof}
  \Cref{eq:speccyce} appears with proof as \cite[eq.~1.4.18]{bll}.

  Proof of \cref{eq:speccyct} proceeds essentially identically to that of \cref{eq:speclint}.
  Once again, we note that the combined action of $\tau$ and a permutation of cycle type $\lambda \vdash n$ can only fix a cyclic order if $\lambda$ satisfies very strong constraints.
  If $n$ is odd, $\lambda$ must consist of $\floor{\sfrac{n}{2}}$ $2$'s and a single $1$, as illustrated in \cref{fig:speccyctodd}.

\begin{figure}[ht]
  \centering
  \begin{tikzpicture}
    \node[style=cnode] at (0:2) (a) {$a$};
    \node[style=cnode] at (72:2) (b) {$b$};
    \node[style=cnode] at (144:2) (c) {$c$};
    \node[style=cnode] at (216:2) (d) {$d$};
    \node[style=cnode] at (288:2) (e) {$e$};

    \path[style=diredge, bend right]
    (a) edge (b)
    (b) edge (c)
    (c) edge (d)
    (d) edge (e)
    (e) edge (a);

    \path[<->, dashed]
    (b) edge (e)
    (c) edge (d);
  \end{tikzpicture}
  \caption{Schematic of a permutation (dashed arrows) which reverses the direction of a cyclic order (solid arrows) of odd length}
  \label{fig:speccyctodd}
\end{figure}
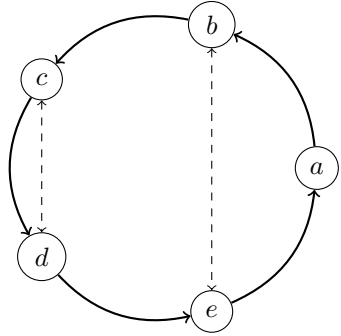

  If, on the other hand, if $n$ is even, $\lambda$ may consist either of $\sfrac{n}{2}$ $2$'s or $\pbrac*{\sfrac{n}{2} - 1}$ $2$'s and two $1$'s, as illustrated in \cref{fig:speccycteven}.

\begin{figure}[ht]
  \centering
  \begin{subfigure}[b]{.4\linewidth}
    \begin{tikzpicture}
      \node[style=cnode] at (0:2) (a) {$a$};
      \node[style=cnode] at (60:2) (b) {$b$};
      \node[style=cnode] at (120:2) (c) {$c$};
      \node[style=cnode] at (180:2) (d) {$d$};
      \node[style=cnode] at (240:2) (e) {$e$};
      \node[style=cnode] at (300:2) (f) {$f$};

      \path[style=diredge, bend right]
      (a) edge (b)
      (b) edge (c)
      (c) edge (d)
      (d) edge (e)
      (e) edge (f)
      (f) edge (a);

      \path[<->, dashed]
      (e) edge (f)
      (d) edge (a)
      (c) edge (b);
    \end{tikzpicture}
    \caption{All $2$'s}
    \label{fig:speccycteven2}
  \end{subfigure}
  \begin{subfigure}[b]{.4\linewidth}
    \begin{tikzpicture}
      \node[style=cnode] at (0:2) (a) {$a$};
      \node[style=cnode] at (60:2) (b) {$b$};
      \node[style=cnode] at (120:2) (c) {$c$};
      \node[style=cnode] at (180:2) (d) {$d$};
      \node[style=cnode] at (240:2) (e) {$e$};
      \node[style=cnode] at (300:2) (f) {$f$};

      \path[style=diredge, bend right]
      (a) edge (b)
      (b) edge (c)
      (c) edge (d)
      (d) edge (e)
      (e) edge (f)
      (f) edge (a);

      \path[<->, dashed]
      (b) edge (f)
      (c) edge (e);
    \end{tikzpicture}
    \caption{Two $1$'s}
    \label{fig:speccycteven1}
  \end{subfigure}
  \caption{Schematics of permutations (dashed arrows) which reverse the direction of a cyclic order (solid arrows) of even length}
  \label{fig:speccycteven}
\end{figure}
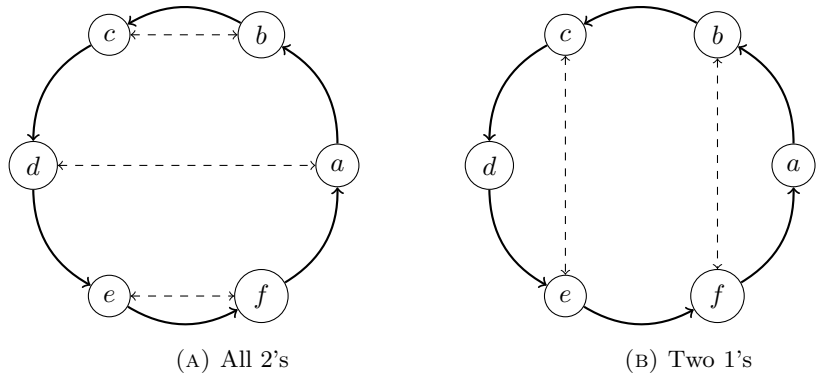

  Counting arguments analogous to those in the previous proof then yield the desired result by explaining the coefficients $\frac{1}{2}$ and $1$ in \cref{eq:speccyct}.
\end{proof}

The $\symgp{2}$-cycle index $\gci{\symgp{2}}{\speccyc}$ of $\speccyc$ is available in Sage \cite{sage}:

\begin{description}
\item[module:]
  \code{sage.combinat.species.group\string_cycle\string_index\string_series\string_library}

\item[class:]
   \code{CyclicOrderWithReversalGroupCycleIndex()}
\end{description}

\subsection{Linear $k$-orders with arbitrary interchange}
\label{s:ordint}
Fix $k \in \mathbf{N}$ and a permutation group $\Gamma \subgp \symgp{k}$.
The species $\speclin[k]$ of linear $k$-orders admits a natural action of $\Gamma$ which permutes the \enquote{slots} of each $\speclin[k]$-structure.
For example, $\speclin[k] \sbrac{\pbrac{1 3 2}} \pbrac{\sbrac{A, B, C}} = \sbrac{B, C, A}$.
Thus, $\speclin[k]$ is a $\Gamma$-species with respect to such an action of any $\Gamma \subgp \symgp{k}$.

\begin{theorem}
  \label{thm:ordintgci}
  If $\Gamma \subgp \symgp{k}$, the $\Gamma$-cycle index of the $\Gamma$-species $\speclin[k]$ of linear $k$-orders with interchange group $\Gamma$ is given by
  \begin{equation}
    \label{eq:ordintgci}
    \gcielt{\Gamma}{\speclin[k]}{\gamma} = p_{\gamma}
  \end{equation}
  where $p_{\gamma} = p_{1}^{\gamma_{1}} p_{2}^{\gamma_{2}} \dots$ for $\gamma_{i}$ the number of $i$-cycles in $\gamma$ as a permutation of $\sbrac{k}$.
\end{theorem}

\begin{proof}
  Per \cref{eq:gcidefpart}, the coefficient of $p_{\lambda}$ in $\gcielt{\Gamma}{\speclin[k]}{\gamma}$ is equal to $1 / z_{\lambda}$ times the number of $\speclin[k]$-structures fixed by the combined action of $\gamma$ and a label permutation of cycle type $\lambda$.
  If $\gamma$ is not of cycle type $\lambda$, this is clearly $0$; otherwise, the number of linear $k$-orders which are fixed by the combined action of $\gamma$ and some permutation of cycle type $\lambda$ is clearly $z_{\lambda}$.
  \Cref{eq:ordintgci} follows.
\end{proof}

\section{Examples of $\Gamma$-species enumeration}
\label{sec:gspecenum}
\subsection{Graphs with complementation}
\label{sec:compgraph}
Let $\specgraph$ denote the species of simple graphs.
It is well-known (see \cite{bll}) that
\begin{equation}
  \label{eq:graphfunccomp}
  \specgraph = \speclin[2] \pbrac*{\specset} \funccomp \pbrac*{\specset[2] \cdot \specset}.
\end{equation}
The species $\specgraph$ admits a natural action of $\symgp{2}$ in which the nontrivial element $\tau$ sends each graph to its complement.
By construction, if we give $\specset$ the trivial $\symgp{2}$-action and $\speclin[2]$ the order-reversing action of \cref{sec:ordrev}, then \cref{eq:graphfunccomp} may be read as an isomorphism of $\symgp{2}$-species.

The quotient of $\specgraph$ under this action of $\symgp{2}$ is the species $\specgraphc = \quot{\specgraph}{\symgp{2}}$ of \enquote{complementation classes}---that is, of pairs of complementary graphs on the same vertex set.
Additionally, per \cref{thm:polyagspec}, $\gcieltvars{\specgraph}{\symgp{2}}{\tau}{x, x^{2}, x^{3}, \dots}$ is the ordinary generating function for unlabeled self-complementary graphs.
This analysis is conceptually equivalent to that given by Read \cite{read:scgraphs}, although of course the $\Gamma$-species approach may be written much more compactly.

Sage code to enumerate complementarity classes of graphs is available from the author on request.
It is necessary to implement the $\symgp{2}$-cycle index of $\specgraph$ manually, which results in code of considerable length.

\subsection{Digraphs with reversal}
\label{sec:digraph}
Let $\specdigraph$ denote the species of directed graphs.
In natural language, \enquote{a digraph is a subset of the set of ordered pairs of vertices}, so in the algebra of species we conclude that
\begin{equation}
  \label{eq:digraphfunccomp}
  \specdigraph = \specsubs \funccomp \pbrac*{\speclin[2] \cdot \specset}.
\end{equation}

The species $\specdigraph$ admits a natural action of $\symgp{2}$ in which the nontrivial element $\tau$ reverses the direction of all edges.
By construction, if we give $\specsubs$ and $\specset$ trivial $\symgp{2}$-actions and $\speclin[2]$ the order-reversing action of \cref{sec:ordrev}, then \cref{eq:digraphfunccomp} may be read as an isomorphism of $\symgp{2}$-species.

The quotient of $\specdigraph$ under this action of $\symgp{2}$ is the species $\specdigraphc = \quot{\specdigraph}{\symgp{2}}$ of \enquote{conversity classes} of digraphs---that is, of digraphs identified with their converses.
In light of \cref{eq:btspec}, we can compute the cycle index of $\specdigraphc$ using the Sage code appearing in \cref{list:condg}.
We note that the results in \cref{tab:condg} agree with those given in \cite[A054933]{oeis}, although in this case our method is much less computationally-efficient than others referenced there.

\begin{lstlisting}[float=htbp,caption=Sage code to compute numbers of conversity classes of digraphs, label=list:condg, language=Python, texcl=true, xleftmargin=3em]
from sage.combinat.species.group_cycle_index_series import GroupCycleIndexSeriesRing
from sage.combinat.species.library import SetSpecies, SubsetSpecies
from sage.combinat.species.group_cycle_index_series_library import LinearOrderWithReversalGroupCycleIndex

S2 = SymmetricGroup(2)
GCISR = GroupCycleIndexSeriesRing(S2)

P = GCISR(SubsetSpecies().cycle_index_series())
E = GCISR(SetSpecies().cycle_index_series())
L2 = LinearOrderWithReversalGroupCycleIndex().restricted(min=2,max=3)

D = P.functorial_composition(L2*E)

print D.quotient().isotype_generating_series().counts(6)
\end{lstlisting}

\begin{table}[htbp]
  \centering
  \begin{tabular}{l l}
    \toprule
    $n$ & $\specdigraphc[n]$ \\
    \midrule
    0 & 1 \\
    1 & 1 \\
    2 & 3 \\
    3 & 13 \\
    4 & 144 \\
    5 & 5158 \\
    6 & 778084 \\
    \bottomrule
  \end{tabular}
  \caption{Number $\specdigraphc[n]$ of isomorphism classes of conversity classes of digraphs with $n$ vertices}
  \label{tab:condg}
\end{table}

Again, per \cref{thm:polyagspec}, $\gcieltvars{\specdigraph}{\symgp{2}}{\tau}{x, x^{2}, x^{3}, \dots}$ is the ordinary generating function for unlabeled self-complementary digraphs.
This analysis is conceptually analogous to that given by Harary and Palmer \cite[\S 6.6]{hp:graphenum}, but, again, the $\Gamma$-species account is much more compact.

\subsection{Binary trees with reversal}
\label{sec:bintrees}
Let $\specbintree$ denote the species of binary rooted trees.
It is a classical result that
\begin{equation}
  \specbintree = 1 + \specsing + \specsing \cdot \speclin[2] \pbrac{\specbintree - 1}
  \label{eq:btspec}
\end{equation}
for $\specsing$ the species of singletons.

$\specbintree$ admits a natural $\symgp{2}$-action whose nontrivial element reflects each tree across the vertical axis, and we may treat it as a $\symgp{2}$-species with respect to this action.
Thus, \cref{eq:btspec} also holds as an isomorphism of $\symgp{2}$-species with $\specsing$ equipped with the trivial $\symgp{2}$-action and $\speclin[2]$ equipped with the order-reversing action from \cref{sec:ordrev}.

The quotient of $\specbintree$ under this action of $\symgp{2}$ is the species $\specbintreer = \quot{\specbintree}{\symgp{2}}$ of \enquote{reversal classes} of binary trees---that is, of binary trees identified with their reverses.
In light of \cref{eq:btspec}, we can compute the cycle index of $\specbintreer$ using the Sage code appearing in \cref{list:btr}.
We note that the results in \cref{tab:btr} agree with those given in \cite[A007595]{oeis}.

\begin{lstlisting}[float=htbp,caption=Sage code to compute numbers of reversal classes of binary trees, label=list:btr, language=Python, texcl=true, xleftmargin=3em]
from sage.combinat.species.group_cycle_index_series import GroupCycleIndexSeriesRing
from sage.combinat.species.library import SingletonSpecies
from sage.combinat.species.group_cycle_index_series_library import LinearOrderWithReversalGroupCycleIndex

S2 = SymmetricGroup(2)
GCISR = GroupCycleIndexSeriesRing(S2)

X = GCISR(SingletonSpecies().cycle_index_series())
L2 = LinearOrderWithReversalGroupCycleIndex().restricted(min=2,max=3)

BT = GCISR(0)
BT.define(1+X+X*L2(BT - 1))

print BT.quotient().isotype_generating_series().counts(20)
\end{lstlisting}

\begin{table}[htbp]
  \centering
  \begin{tabular}{l l}
    \toprule
    $n$ & $\specbintreer[n]$ \\
    \midrule
    0 & 1 \\
    1 & 1 \\
    2 & 1 \\
    3 & 3 \\
    4 & 7 \\
    5 & 22 \\
    6 & 66 \\
    7 & 217 \\
    8 & 715 \\
    9 & 2438 \\
    \bottomrule
  \end{tabular}
  \caption{Number $\specbintreer[n]$ of isomorphism classes of binary trees up to reversal with $n$ internal vertices}
  \label{tab:btr}
\end{table}

\subsection{$k$-ary trees with interchange}
Let $\specrtree[k]$ denote the species of $k$-ary rooted trees---that is, rooted trees where each node has $k$ linearly-ordered child trees.
Any $\Gamma \subgp \symgp{k}$ acts naturally on $\specrtree[k]$; an element $\gamma \in \Gamma$ acts on a tree $T$ by applying $\gamma$ to the linear order on each node's children.
Thus, $\specrtree[k]$ is a $\Gamma$-species with respect to this action, and it satisfies
\begin{equation}
  \label{eq:karytree}
  \specrtree[k] = 1 + \specsing + \specsing \cdot \speclin[k] \pbrac*{\specrtree[k]}
\end{equation}
for $\specsing$ the species of singletons with the trivial $\Gamma$-action and $\speclin[k]$ the $\Gamma$-species of linear $k$-orders from \cref{s:ordint}.
Its quotient is the species $\quot{\specrtree[k]}{\Gamma}$ of $\Gamma$-equivalence classes of $k$-ary trees.

\subsection{Paths and polygons}
\label{sec:pathpoly}
Recall from \cref{sec:ordrev} the $\symgp{2}$-species $\speccyc$ of cyclic orders and $\speclin$ of linear orders with reversal.
Their quotients are, respectively, the species $\specpoly = \quot{\speccyc}{\symgp{2}}$ of \enquote{necklaces} and $\specpath = \quot{\speclin}{\symgp{2}}$ of \enquote{paths}.
This species of polygons is also studied using similar methods in \cite[\S 3]{lab:smallcard}.

We can compute the cycle indices of $\specpath$ and $\specpoly$ using the Sage code appearing in \cref{list:pathpoly}.
Of course, there is only one $\specpath$-structure and one $\specpoly$-structure for each $n$, so we do not print the results here.

\begin{lstlisting}[float=htbp,caption=Sage code to compute numbers of paths and polygons, label=list:pathpoly, language=Python, texcl=true, xleftmargin=3em]
from sage.combinat.species.group_cycle_index_series_library import LinearOrderWithReversalGroupCycleIndex, CyclicOrderWithReversalGroupCycleIndex

L = LinearOrderWithReversalGroupCycleIndex()
Path = L.quotient()
print Path.generating_series().counts(8)
print Path.isotype_generating_series().counts(8)

C = CyclicOrderWithReversalGroupCycleIndex()
Poly = C.quotient()
print Poly.generating_series().counts(8)
print Poly.isotype_generating_series().counts(8)
\end{lstlisting}

\subsection{Bicolored graphs}
\label{sec:bcgraphs}
The species $\specgraphbc$ of properly $2$-colored graphs admits a structural $\symgp{2}$-action in which the nontrivial element interchanges colors.
In \cite{bpblocks}, Gessel and the author compute the $\symgp{2}$-cycle index of $\specgraphbc$ from first principles, then apply the results of \cref{sec:gspeciesalg} and some structural results to enumerate unlabeled bipartite blocks.

\subsection{$k$-trees}
\label{sec:ktrees}
In \cite{agdktrees}, the author enumerates $k$-trees using the theory of $\Gamma$-species.
To achieve this, we introduce the notion of an \emph{oriented} $k$-tree (which is a $k$-tree together with a choice of cyclic ordering of the vertices of each $\pbrac{k+1}$-clique, subject to a compatibility condition).
We then recast the problem as one of enumerating \emph{orbits} of oriented $k$-trees under a suitable action of $\symgp{k}$ and calculate the relevant $\symgp{k}$-cycle index series using recursive structure theorems.

In \cite{gdgktrees}, Gessel and the author simplify this approach, using colorings instead of cyclic orderings to break the symmetries of $k$-trees.
The results in that work are phrased in the language of generating functions, without explicit reference to species-theoretic cycle indices.

\subsection*{Acknowledgments}
The author is grateful to Ira Gessel for several helpful conversations, especially concerning the proof of \cref{thm:polyagspec}.

\end{document}